\documentclass[11pt,letterpaper]{article}
\usepackage{fullpage}
\usepackage{amsmath}
\usepackage{amsfonts}
\usepackage{amssymb}
\usepackage{tikz-cd} 
\usepackage{titlesec}
\usepackage{tasks}
\usepackage{mathabx}
\usepackage{color}
\usepackage{mathrsfs}
\usepackage{hyperref}
\usepackage[all]{xy}
\usepackage{graphicx}
\usepackage{cite}
\usepackage{caption}
\usepackage{graphicx, subfig}
\usepackage{enumerate}
\usepackage[T1]{fontenc}
\usepackage[utf8]{inputenc}
\usepackage{authblk}

\title{Quantum $K$-theory and  $q$-Difference equations}
\author[1]{Yongbin Ruan}
\author[2]{Yaoxiong Wen}
\affil[1]{Institute of Advanced Study for Mathematics, Zhejiang University. \authorcr Email: ruanyb@zju.edu.cn}
\affil[2]{Beijing International Center For Mathematical Research, Peking University. \authorcr Email: y.x.wen.math@gmail.com}

\date{}

\newtheorem{Theorem}{Theorem}[section]
\newtheorem{Lemma}{Lemma}[section]
\newtheorem{Proposition}{Proposition}[section]

\newtheorem{Definition}{Definition}[section]
\newtheorem{Example}{Example}[section] 

\newtheorem{Remark}{Remark}[section]

\newenvironment{proof}{{\noindent\it Proof}\quad}{\hfill $\square$\par}

\begin{document}
\maketitle
\begin{abstract}
This is a set of lecture notes for the first author's lectures on the difference equations in 2019 at the Institute of Advanced Study for Mathematics at Zhejiang University. We focus on explicit computations and examples. The convergence of local solutions is discussed.
\end{abstract}
\tableofcontents
\clearpage

\section{Introduction}
The linear differential equation appears naturally in quantum cohomology and mirror symmetry as the Picard-Fuch equation of period integral. It has been studied extensively throughout history. Its close cousin-difference equation is an ancient topic of mathematics as old as the linear differential equation. However, it did not receive nearly enough attention. During the last decade, it appears as the analogous of Picard-Fuch equation in quantum $K$-theory. The latter enjoys a revival due to the recent realization that quantum $K$-theory is a 3d TQFT. It was well-known that quantum cohomology is a 2d TQFT. 

Compared to its more famous cousin, there is a lack of literature on difference equations, which slows down the development of the subject. This set of notes is an attempt to improve the situation. Instead of developing the general theory, we focus on the explicit computation of the solutions. We claimed no originality of material and made no attempt to complete references, for which the author apologize.
	
\subsection{Linear difference equation}

Our main consideration is following equation
The equation
\begin{align}
\sum_{i=0}^n a_i(z)f(q^iz) = 0,	
\end{align}
in which the $a_i(z)$ are known functions of the complex variable $z, q$  and $|q| > 1$, is called the homogeneous linear ordinary $q$-difference equation of the $n$-th order. 

In early 2000, Y.P. Lee \cite{Lee} and Givental \cite{Givental} introduced the $K$-theoretic Gromov--Witten theorem, in which the $q$-difference equations play the role of Picard-Fuch equation in quantum cohomology. More precisely, the quantum $K$-theoretic $I$-function  satisfy certain $q$-difference equations.

For example, let's consider projective space $\mathbb{P}^N$, the modified $K$-theoretic $I$-function is of the form
\begin{align}
\widetilde{I^K_{\mathbb{P}^N}} = 	P^{\ell_q(z)}\sum_{d=0}^{\infty} \frac{z^d}{\prod_{k=1}^{d} (1-Pq^k)^{N+1} } ,
\end{align}
where $P=\mathcal{O}(-1)$ on $\mathbb{P}^N$. Denoted by $\sigma_q:=q^{z\partial_z}$ the difference operator shifting $z^k$ by $q^kz^k$. Since $(1-P)^{N+1}=0$, then the $K$-theoretic $I$-function of $\mathbb{P}^N$ satisfies the following degree $N+1$ difference equation
\begin{align}
\left[ (1-\sigma_q)^{N+1} -z  \right] \widetilde{I^{K}_{\mathbb{P}^N}} = 0.	
\end{align}

In 2018, Ruan-Zhang \cite{RZ} introduced a key new feature in quantum $K$-theory, i.e., the level structure, which is now well-understood to correspond to Chern-Simons term in so called 3d $\mathcal{N}=2$ theory  in physics\cite{HP20,UY20} . It plays an essential role in 3d theories \cite{RWZ} and has effects on the difference equation as follows.

Still taking projective space $\mathbb{P}^N$ as an example, let's consider the level structure with respect to standard representation of $\mathbb{C}^*$ of level $l$, see \cite{RZ} for details. Then the modified $I$-function with level structure is
\begin{align} \label{I-func-with-level}
\widetilde{I^{K, l}_{\mathbb{P}^N}} = 	P^{\ell_q(z)}\sum_{d=0}^{\infty} \frac{\left(P^{d}q^{\frac{d(d-1)}{2}} \right)^l z^d}{\prod_{k=1}^{d} (1-Pq^k)^{N+1} },
\end{align}
and it satisfies 
\begin{align}\label{dif-I-level}
\left[ (1-\sigma_q)^{N+1} -z  \sigma_q^l  \right] \widetilde{I^{K,St}_{\mathbb{P}^N}} = 0.	
\end{align}

Moreover, we could consider hypersurfaces inside projective space. Among them, the quintic 3-fold plays an essential role. Let's denote the quintic 3-fold by $X$, which could be realized as a degree 5 hypersurface in the projective space $\mathbb{P}^4$. By using the quantum Lefschetz hyperplane theorem \cite{Givental15}, we have
\begin{align}
\widetilde{I^K_X}(z,q)= P^{\ell_q(z)}\sum_{d=0}^{\infty} \frac{\prod_{k=1}^{5d}(1-P^5q^k)}{\prod_{k=1}^{d}(1-Pq^k)^5} z^d \label{K-I-function},
\end{align}
where we still use $P$ to denote $\mathcal{O}(-1)$ on $\mathbb{P}^4$. Since $(1-P)^5=0$ in $K(\mathbb{P}^4)$, then (\ref{K-I-function}) satisfies the following difference equation
\begin{align} 
\left[	(1-\sigma_q)^5 - z \prod_{k=1}^5 (1-q^{k}\sigma_q^5) \right] \widetilde{I^K_X}(q,z)=0, \label{quintic}    
\end{align}
It's a degree 25 difference equation! 

Note that the difference equation (\ref{dif-I-level}) behaves well when level $0 \leq l \leq N+1$, in these cases, the difference equations are regular singular (Definition \ref{reg-sing}). And for difference equations (\ref{dif-I-level}) with $l \geq N+1$ and the difference equation (\ref{quintic}), they are called irregular singular (Definition \ref{reg-sing}), their solutions are not as good as regular singular cases.

Finding the difference equations that $I$-functions satisfy is significant. One formulation of 3d mirror symmetry is to interchange the quantum/equivariant parameters of mirror pair \cite{AHKT, DT, RWZ}. Usually, $I$-functions are very complicated, and it is difficult to do it directly. One often accomplishes it by analyzing its q-difference equation.

In this paper, our primary goal is to find solutions of $q$-difference equations via the Frobenius method. We will use many concrete examples to demonstrate how it works in both regular singular cases and irregular singular cases. We focus on computations rather than general results. Among these examples, we will see a lot of modular forms! The paper is organized as follows. Section 2 reviews some basic definitions of difference equations and some general results about regular singular cases. In section 3, we introduce the Frobenius method and apply it to some examples. We end this section with a discussion about the convergence of the solutions. We will show that under certain conditions, all the power-series solutions in regular singular cases are convergent. Section 4 deals with irregular singular cases; we start by showing how the general technique works and then apply it to some examples, including the difference equations for the quintic and projective space with level structures. We also end up with a discussion about the convergence of the specific solutions in the irregular cases.

\subsection{Acknowledgements}
These lectures were held at the Institute for Advanced Study in Mathematics at Zhejiang University. We express our special thanks to the institute for its wonderful environment and support. The second author would like to thank Prof. Bohan Fang, Prof. Huijun Fan, and Prof. Shuai Guo for their helpful support during the visit.

\section{A brief review of general theory}
In this section, we review some basic  in the theory of $q$-difference equations. The main references are \cite{hardouin:hal-01959879}, \cite{2019arXiv191100254R} and \cite{Sauloy00}.  
\vspace{0.3cm}

\noindent {\bf{Notations}}. Here are some standard notations of general use:
\begin{itemize}
\item[--] $z$ and $q$ are complex variables and $|q|>1$,
\item[--] $\mathbb{C}(z)$ is the field of rational fractions over $\mathbb{C}$,
\item[--] $\mathbb{C}(\{z\})$ is the field of meromorphic germs at 0, is the quotient field of $\mathbb{C} \{ z \}$,
\item[--] $\mathbb{C}(( z ))$ is the field of Laurent formal power series, is the quotient field of $\mathbb{C}[[z]]$,
\item[--] $\mathcal{M}(\mathbb{C})$ is the field of meromorphic functions on $\mathbb{C}$,
\item[--] $\mathcal{M}(\mathbb{C}^*)$ is the field of meromorphic functions on $\mathbb{C}^*$,
\item[--] $\mathcal{M}(\mathbb{C}^*,0)=\mathbb{C}\{z , z^{-1}\}$ is the space of convergent Laurent series defined on a punctured disk at $z=0$,
\item[--] $\mathcal{M}\left(\mathbb{E}_{q}\right)$ is the field of meromorphic functions on elliptic curve $\mathbb{E}_q=\mathbb{C}^*/\{q^n \mathbb{C}^*, n\in \mathbb{Z}\}$, i.e, the field of
elliptic functions.
\end{itemize}

\begin{Definition}
A difference field is a pair $(K,{\sigma})$, where $K$ is a field and $\sigma$ is a field automorphism of $K$.	
\end{Definition}

\begin{Example}
We will focus on the fields in the above notations, 
\begin{align}
\mathbb{C}(z) \subset \mathcal{M}(\mathbb{C}) \subset \mathcal{M}(\mathbb{C},0) = \mathbb{C}(\{ z\})	\ and \ \mathcal{M}(\mathbb{C}^*) \subset \mathcal{M}(\mathbb{C}^*,0),
\end{align}
they are all endowed with the $q$-shift operator $\sigma_q := q^{z \partial_{z}}: f(z) \mapsto f(qz)$.	Usually, we denote the field of constants of the difference field $(K, \sigma_q)$ as $K^{\sigma_q}$. For example, $\mathcal{M}(\mathbb{C}^*)^{\sigma_q}=\mathcal{M}\left(\mathbb{E}_{q}\right)$. This is the main reason that the modular form such as elliptic function
appears naturally in the theory of $q$-difference equation.
\end{Example}

\subsection{Difference equations}
The $q$-difference equation is as follows
\begin{align}
a_n(z,q)(\sigma_q)^nf+	a_{n-1}(z,q)(\sigma_q)^{n-1}f+ \cdots + a_{0}(z,q)f = 0,
\end{align}
with $|q| < 1$ and $a_i(z,q)$ are meromorphic functions. Let $P:=a_n(z,q)(\sigma_q)^n+	a_{n-1}(z,q)(\sigma_q)^{n-1}+ \cdots + a_{0}(z,q)$. Then the vectorialisation trick shows that:
\begin{align}
P \cdot f=0  \Longleftrightarrow \sigma_q X_f = A_P X_f,	
\end{align}
where
\begin{align}
X_f:=\left(\begin{array}{c}
f(z,q) \\
\sigma_q f(z,q) \\
\vdots \\
\left(\sigma_q\right)^{n-1} f(z,q)
\end{array}\right)	\ and \ A_P:= \left(\begin{array}{ccccc}
0 & 1 & 0 & \cdots & 0 \\
\vdots & 0 & \ddots & & 0 \\
0 & \cdots & \cdots & \cdots & 1 \\
-\frac{a_{0}}{a_{n}} & -\frac{a_{1}}{a_{n}} & \cdots & \cdots & -\frac{a_{n-1}}{a_{n}}
\end{array}\right).
\end{align}
For this reason, the study of scalar linear difference equations boils down to that of difference systems. In the following, we focus on the study of difference systems and in the next section, we will focus on the linear difference equations in order to apply the Frobenius method.

\begin{Remark}
We see how to obtain a difference systems from a linear difference equation, in some sense, the converse is true.	
\end{Remark}

\begin{Theorem}[\cite{hardouin:hal-01959879}, Theorem 2.4.8 (Cyclic vector lemma)] 
Every $A \in GL_n(K)$, where $K$ is a field, is equivalent over $K$ to some $A_P$, where $P$ is a difference operator. 
	
\end{Theorem}

\begin{Definition}
Let $\left(E_{q}\right):\sigma_q X_{q}(z)=A_{q}(z) X_{q}(z)$  be a  $q$-difference system, with $A_{q} \in \mathrm{Mat}_{n \times n}(K)$. We define the solution space of this $q$-difference equation by
\begin{align}
\operatorname{Sol}\left(E_{q}\right)=\left\{X_{q} \in K^{n} \mid \sigma_q X_{q}(z)=A_{q}(z) X_{q}(z)\right\} \label{Sol}.
\end{align}	
\end{Definition}

\begin{Remark}
From now on, we will focus on the local solutions at $z=0$, the results will also hold for $z=\infty$, the reason why we don't consider solutions at other singular points is that: if a function $f(z)$ is a solution of a $q$-difference equation $\sigma_q f(z)=a(z)f(z)$ and has a singularity at some $z_0 \neq 0,\infty$, then $f(z)$ has a singularity at any complex number $z_0q^k$.	
\end{Remark}

\begin{Proposition}[\cite{hardouin:hal-01959879}, Theorem 2.3.1 p.118]
Let $\left(E_{q}\right): \sigma_q X_{q}(z)=A_{q}(z) X_{q}(z)$ be a $q$-difference system. Then, we have
\begin{align}
	\operatorname{dim}_{K^{\sigma_q}}\left(\operatorname{Sol}\left(E_{q}\right)\right) \leq \operatorname{rank}\left(A_{q}\right).
\end{align}
Even more, if we consider solutions in a bigger space, i.e. the extension of the difference field $(K, \sigma_q)$ (is a difference field $(K^{\prime}, \sigma_q^{\prime})$ such that $K \subset K^{\prime}$ and $\sigma_q^{\prime}|_{K} = \sigma_q $), we still have
\begin{align}
	\operatorname{dim}_{(K^{\prime})^{\sigma_q}}\left(\operatorname{Sol}\left(E_{q}, K^{\prime}\right)\right) \leq \operatorname{rank}\left(A_{q}\right).
\end{align}
\end{Proposition}


\begin{Definition}
Let $\sigma_q X_{q}(z)=A_{q}(z) X_{q}(z)$ be a $q$-difference system. The fundamental solution is a family $\mathcal{X}_{q}= (f_1, \ldots, f_n)$ in  $\operatorname{Sol}\left(E_{q}\right)$, such that the determinant of wronskian matrix is not 0. The wronskian matrix is defined as:
\begin{align}
	W_n(f_1,\ldots, f_n) := 
\left(\begin{array}{ccccc}
f_1 & f_2 &  \cdots & f_n \\
\sigma_qf_1 & \sigma_q f_2  & \cdots & \sigma_q f_n \\
\vdots & \vdots & \ddots  & \vdots \\
\sigma_q^{n-1}f_1 & \sigma_q^{n-1}f_2 & \cdots & \sigma_q^{n-1}f_n 
\end{array}\right).
\end{align}
Especially, in the case of $K=\mathbb{C}\{z, z^{-1} \}$, a fundamental solution of this system is an invertible matricial solution $\mathcal{X}_{q} \in \mathrm{GL}_{n}\left(\mathbb{C}\left\{z, z^{-1}\right\}\right)$ such that $\sigma_q \mathcal{X}_{q}(z)=A_{q}(z) \mathcal{X}_{q}(z)$.	
\end{Definition}

\begin{Lemma}[\cite{hardouin:hal-01959879}, Lemma 2.3.3 (Wronskian lemma)]
Let $f_1, \ldots, f_n \in K$ and denote the determinant of their wronskian matrix as 
\begin{align}
w_n:= \det W_n(f_1,\ldots,f_n).	
\end{align}
Then $w_n(f_1,\ldots,f_n)=0$, if and only if, $f_1,\ldots,f_n$ are linearly dependent over $K^{\sigma_q}$.
\end{Lemma}

\begin{Proposition}[\cite{hardouin:hal-01959879}, Proposition 2.4.4. p.120] 
Let $\mathcal{X}_q \in GL_n \left( K \right)$ a fundamental matricial solution of (\ref{Sol}). Then:
\begin{align}
	\operatorname{Sol}\left(E_{q}\right) =\{ \mathcal{X}_qC \ \big{|} \ C \in \left(K^{\sigma_q} \right)^n \} = \mathcal{X}_q \cdot \left(K^{\sigma_q} \right)^n.
\end{align}	
\end{Proposition}

\subsection{Regular singular $q$-difference equations}
Usually, we say global study of $q$-difference equations if we take $K= \mathbb{C}(z)$, local analytic study if we take $K=\mathbb{C}(\{z\})$ and formal study if we take $K=\mathbb{C}((z))$. In the following, we shall look for solutions in $K=\mathcal{M}(\mathbb{C}^*,0)=\mathbb{C}\{z,z^{-1} \}$ due to many consequences on the shape of the analytical theory. 

\begin{Definition}
Let $\sigma_q X_{q}(z)=A_{q}(z) X_{q}(z)$ be a $q$-difference system. Consider a matrix $P_{q} \in \mathrm{GL}_{n}\left(\mathbb{C}\left\{z, z^{-1}\right\}\right)$. The gauge transform of the matrix $A_{q}$ by the gauge transformation $P_q$ is the matrix
\begin{align}
	P_{q} \cdot\left[A_{q}\right]:=\left(\sigma_q P_{q}\right) A_{q} P_{q}^{-1}.
\end{align}
A second $q$-difference system $\sigma_q X_{q}(z)=B_{q}(z) X_{q}(z)$ is said to be equivalent by gauge transform to the first one if there exists a matrix $P_{q} \in \mathrm{GL}_{n}\left(\mathbb{C}\left\{z, z^{-1}\right\}\right)$ such that
\begin{align}
B_{q}=P_{q} \cdot\left[A_{q}\right].	
\end{align}
\end{Definition}

Let us introduce the regular and regular singular $q$-difference equations.
\begin{Definition}
A system $\sigma_q X_{q}(z)=A_{q}(z) X_{q}(z)$ is regular if $A_q(0)$ is diagonal and if its eigenvalues are of the form $q^k$ for $k \in \mathbb{Z}_{\geq 0}$ 	
\end{Definition}

\begin{Definition} \label{reg-sing}
A system $\sigma_q X_{q}(z)=A_{q}(z) X_{q}(z)$ is said to be regular singular at $z=0$ if there exists a $q$-gauge transform $P_{q} \in \mathrm{GL}_{n}\left(\mathbb{C}\left\{z, z^{-1}\right\}\right)$ such that the matrix $\left(P_{q} \cdot\left[A_{q}\right]\right)(0)$ is well-defined and invertible: $P_{q} \cdot\left[A_{q}\right](0) \in \mathrm{GL}_{n}(\mathbb{C})$. Otherwise, we say the system irregular singular.	 
\end{Definition}

Let us give a criteria for when a $q$-difference equation is regular singular at $z=0$.
\begin{Proposition}[\cite{2019arXiv191100254R}, Proposition V.2.1.14.] \label{critieria-regular}
Let $P=\sum_{k}^{n} a_{k}(z,q)\left(\sigma_q \right)^{k}$ be a $q$-difference operator. As we stated before that the $q$-difference equation $P \cdot f (z)=0$ can be vectorised to a $q$-difference system $\sigma_q X_{q}(z)=$ $A_{q}(z) X_{q}(z)$ where $A_{q}(z)$ is the companion matrix of the operator $P$. The resulting $q$ -difference system is
\begin{align}
\sigma_q \left(\begin{array}{c}
f(z) \\
\sigma_q f(z) \\
\vdots \\
\left(\sigma_q \right)^{n-1} f(z)
\end{array}\right)=\left(\begin{array}{ccccc}
0 & 1 & 0 & \cdots & 0 \\
\vdots & 0 & \ddots & & 0 \\
0 & \cdots & \cdots & \cdots & 1 \\
-\frac{a_{0}}{a_{n}} & -\frac{a_{1}}{a_{n}} & \cdots & \cdots & -\frac{a_{n-1}}{a_{n}}
\end{array}\right)\left(\begin{array}{c}
f(z) \\
\sigma_q f(z) \\
\vdots \\
\left(\sigma_q \right)^{n-1} f(z)
\end{array}\right).
\end{align}
 We denote by $val_0 \left(a_{k}\right)$ the $z$-adic valuation of the polynomial $a_{k}$, i.e. the lowest integer $\alpha \in \mathbb{Z} \cup \{+\infty\}$ such that $\left(z^{-\alpha} a_{k}(z)\right)\big{|} _{z=0} \neq 0$. The $q$ -difference system associated to the $q$-difference equation $P\left(\sigma_q \right) f(z)=0$ is regular singular if and only if $val_{0}(a_0(z))-val_{0}(a_n(z))=0$, and for every $k \in \{1, \ldots, r-1\}$, ${val}_{0}(a_k(z))-{val}_{0}(a_n(z)) \geq 0$.		
\end{Proposition}

Let's introduce some special functions which are needed to solve $q$-difference equations. The Jacobi's theta function is defined as follows 
\begin{align}
\theta_{q}(z) &=\sum_{d \in \mathbb{Z}} q^{-\frac{d(d+1)}{2}} z^{d}.
\end{align}
This function satisfies the $q$-difference equation $\sigma_q \theta_{q}(z)=z \theta_{q}(z)$. And it has a famous Jacobi's triple identity 
\begin{align}
	\theta_q(z)= (q^{-1};q^{-1})_\infty (-q^{-1}z;q^{-1})_\infty (-z^{-1};q^{-1})_\infty,
\end{align}
here we use the $q$-Pochhammer symbol $(x;q)_k := \prod_{i=1}^k (1-xq^{i-1})  $. By using Jacobi's theta function, we define the following two special functions.

\begin{Definition}
Let $\lambda_{q} \in \mathbb{C}^{*}$. The $q$-character associated to $\lambda$ is the function $e_{q, \lambda_{q}} \in \mathcal{M}\left(\mathbb{C}^{*}\right)$ defined by
\begin{align}
e_{q, \lambda_{q}}(z)=\frac{\theta_{q}(z)}{\theta_{q}\left(  z/\lambda_{q} \right)} \in \mathcal{M}\left(\mathbb{C}^{*}\right).	
\end{align}
\end{Definition}
The function $e_{q,\lambda_q}$ satisfies the	 $q$-difference equation $\sigma_q e_{q, \lambda_{q}}(z)=\lambda_{q} e_{q, \lambda_{q}}(z)$.

\begin{Definition} \label{q-log}
The $q$-logarithm is the function $\ell_{q} \in \mathcal{M}\left(\mathbb{C}^{*}\right)$ defined by
\begin{align}
\ell_{q}(z)=z\frac{ \theta_{q}^{\prime}(z)}{\theta_{q}(z)}.	
\end{align}	
\end{Definition}
 Since 
\begin{align}
\theta_q(qz) = z \theta_q(z),	
\end{align}
then
\begin{align}
	\frac{\partial}{\partial z} \theta_q(qz) &= q \theta_q^{\prime}(qz)   \\
	                                         &= \theta_q(z) +z\theta_q^{\prime}(z).
\end{align}
So
\begin{align}
qz \frac{\theta^{\prime}_q(qz)}{\theta_q(qz)} &= z^2\frac{\theta^{\prime}_q(z)}{\theta_q(qz)} + z \frac{\theta_q(z)}{\theta_q(qz)}	\\
&= z \frac{\theta^{\prime}_q(z)}{\theta_q(z)} +1,
\end{align}
i.e.,
\begin{align}
\sigma_q \ell_q(z) = \ell_q(z) +1.	
\end{align}

\begin{Remark}
In the literature \cite{Adams}, one also considered $\frac{\log z}{\log q}$, which  satisfies
\begin{align}
\sigma_q \left( \frac{\log z}{\log q} \right)	= \frac{\log z}{\log q} +1,
\end{align}
but it is multi-value. 	In this article, we prefer the single valued function $\ell_q(z)$. For Picard-Fuch equation, one can read off the monodromy
directly from local solution. On the other hand, we can have both single valued and multi-valued solutions for $q$-difference equation. This is related 
to the fact that $q$-difference equation has a rather large the field of constants. The issue of monodromy for $q$-difference equation is rather subtle \cite{Sauloy03}.
\end{Remark}

Now we can state the existence of a fundamental solution for regular singular $q$-difference equations under certain condition.

\begin{Definition}
Consider a regular singular $q$-difference system $\sigma_q X_q(z) = A_q(z)X_q(z) $ and denote by $(\lambda_i)$ the eigenvalues of the matrix $A_q(0)$. This $q$-difference system is said to be non ($q$-)resonant if for every $i \neq j$, we have $\frac{\lambda_i}{\lambda_j} \notin q^{\mathbb{Z} \backslash \{0\} }$, where $	q^{\mathbb{Z}\backslash \{0 \} }:=\left\{q^{k} \mid k \in \mathbb{Z}\backslash \{0\} \right\} \subset \mathbb{C} $.	
\end{Definition}

For a non-resonant system, we can recursively build a gauge transform $F_{q} \in \mathrm{GL}_{n}\left(\mathbb{C}\left\{z, z^{-1}\right\}\right)$ which sends the matrix $A_q(z)$ to the constant matrix $A_q(0)$, for details, see \cite{hardouin:hal-01959879}. Then taking the Jordan-Chevalley decomposition of $A(0)=A_s A_u$ where $A_s$ is semi-simple, $A_u$ is unipotent and $A_s$, $A_u$ commute. 

Since $N=A_u-I_n$ is nilpotent, we can define
\begin{align}
A_u^{\ell_q} := (I_n+N)^{\ell_q} := \sum_{k \geq 0 } \left(\begin{array}{l} \ell_q \\ k \end{array}\right) N^k	 \label{A_u^l},
\end{align}
where
\begin{align}
	\left(\begin{array}{l} \ell_q \\ k \end{array}\right) := \frac{\ell_q(\ell_q-1)\cdots(\ell_q-(k-1))}{k!}, 
\end{align}
note that (\ref{A_u^l}) is actually a finite sum and $A_u^{\ell_q}$ is unipotent, then we have
\begin{align}
\sigma_q A_u^{\ell_q} = A_u	A_u^{\ell_q} = A_u^{\ell_q} A_u.
\end{align}
Thus we set 
\begin{align}
e_{q,A_u} := A_u^{\ell_q}.	
\end{align}

Take a basis change $P$ to diagonalise $A_s=P^{-1}diag(\lambda_i)P$. We define
\begin{align}
e_{q, A_s}:=P^{-1} \operatorname{diag}\left(e_{q, \lambda_{i}}(z)\right) P,	
\end{align}
which satisfies
\begin{align}
\sigma_q e_{q, A_s} = A_s e_{q, A_s} = e_{q, A_s} A_s.	
\end{align}
Then one can check that the product $F_{q}\cdot e_{q, A_s} \cdot e_{q, A_u}=: \mathcal{X}_{q}(z)$ is a fundamental solution of the $q$-difference system $\sigma_q X_{q}(z)=A_{q}(z) X_{q}(z)$. We arrive at the following theorem.

\begin{Theorem}[\cite{Sauloy00}, 1.1.4] \label{fundamental-solution}
	Let $\sigma_q X_{q}(z)=A_{q}(z) X_{q}(z)$ be a regular singular $q$-difference system. Assume that this $q$-difference system is non-resonant. Then, there exists a fundamental solution of $\mathcal{X}_{q} \in G L_{n}\left(\mathbb{C}\left\{z, z^{-1}\right\}\right)$ of this $q$-difference equation expressed with functions $e_{q,\lambda_q}(z)$ and $ \ell_q(z) $.
\end{Theorem}

\section{Local solutions for regular singular cases}
In the last section we introduced some general results about the solutions of regular singular $q$-difference system. However, it still be hard to obtain a explicit formula for a solution. In this section, we will use Frobenius method to construct the solutions. The Forbenius method for linear ordinary $q$-difference equation could be dated back to \cite{Adams} and \cite{Carmichael}. 

In the following, we focus on the computations and use concrete examples to show how the Frobenius method works rather than giving a general result.

\subsection{Frobenius method}
Let's consider the equation
\begin{align}
\sum_{i=0}^n a_i(z) (\sigma_q)^i f(z)=0,	
\end{align}
with
\begin{align}
a_i(z)= a_{i0}+a_{i1}z+a_{i2}z^2 + \cdots	.
\end{align}
Since it is regular singular, then coefficients $a_i(z)$ satisfy conditions in Proposition \ref{critieria-regular}, then we could assume that $a_{00}, a_{n0} \neq 0$.

Let's consider the following equation	
\begin{align}
a_{n0}x^n+ a_{n-1,0} x^{n-1} + \cdots + a_{10}x + a_{00} =0. \label{char-equ-2}
\end{align}
It is called the \emph{characteristic equation}, which plays an important role in constructing solutions. 
\begin{itemize}
\item{\bf Non-resonant case:} suppose the $n$ roots $\{c_1, \ldots, c_n \}$ of (\ref{char-equ-2}) are all distinct and 
\begin{align}
c_i/c_j \notin q^{\mathbb{Z}}, \quad  \forall i \neq j . \label{nonresonant-case} 
\end{align}
Then there exist a set of $n$ power-series solutions of the form:
\begin{align}
	 S_{i}(z,q)=e_{q, c_i} \cdot F_i(z,q),\quad {\rm{where}} \quad F_i(z,q) = \sum_{k=0}^{\infty} f_{ik}(q) z^k, \label{sol-non-renonant}
\end{align}
and for $i=1, \ldots, n$. 
\item{\bf Resonant case:} suppose the $n$ roots are as follows
\begin{align}
	c_i\cdot q^{m_{ij}}, \quad {\rm for} \quad i=1,\ldots, r, \quad j=0, \ldots k_i, \label{resonant-case} 
\end{align}
such that
\begin{itemize}
\item[(i)] $c_i/c_j \notin q^{\mathbb{Z}}$,
\item[(ii)] $0=m_{i0} \leq m_{i1} \leq \ldots \leq m_{ik_i} $, and  $\sum_{i=1}^r (k_i+1)=n $.  
\end{itemize}
Then there are $n$ power-series solutions of the following form
\begin{align}
	S_{i0}(z,q)&=e_{q,c_i}F_{i0}(z,q), \label{sol-resonant-0} \\
	S_{i1}(z,q)&=\ell_q(z) S_{i0}(z,q)+ e_{q,c_iq^{m_{i1}}}F_{i1}(z,q), \label{sol-resonant-1} \\
	           & \ldots \nonumber \\
  S_{ik_i}(z,q)&=\ell_q(z) S_{i,k_i-1}(z,q)+ e_{q,c_iq^{m_{ik_i}}}F_{ik_i}(z,q), \label{sol-resonant-ki}
\end{align}
where $i=1, \ldots, r$ and $j=0,\ldots,k_i$, and 
\begin{align}
F_{ij}(z,q)=\sum_{k=0}^\infty	f_{ijk}(q)z^k
\end{align}	
\end{itemize}

Thus, in regular singular cases, there exist a set of $n$ power-series solutions. In the following, we use concrete examples to show how the Frobenius method works. Furthermore, in the last of this section, we show that these $n$ power-series solutions provide a complete set of solutions analytic in the vicinity of the origin under certain conditions.


\begin{Example} \label{Example-1}
Consider the following degree 2 regular singular difference equation
\begin{align}
\left[ a(z,q) \sigma_q^2 + b(z,q) \sigma_q + d(z,q) \right]f=0.	
\end{align}
From the Proposition \ref{critieria-regular}, we know
\begin{align}
val_0 (a(z,q))-val_0(d(z,q))=0 \ and \ val_0(b(z,q)) - val_0(a(z,q)) \geq 0. \label{val-case-1}
\end{align}
For simplicity, we assume $a(z,q)=1$,
\begin{align}
	\sigma_q^2 f + b(z,q) \sigma_q f  + d(z,q)f=0,
\end{align}
where 
\begin{align}
b(z,q)=\sum_{n=0}^{\infty} b_n(q)z^n, \  d(z,q)=\sum_{n=0}^{\infty}	d_n(q)z^n, \ d_0 \neq 0.
\end{align}

We are looking for the solution of the form 
\begin{align}
\sum_{n=0}^{\infty} f_n(q)z^{n+r},	
\end{align}
thus we have
\begin{align}
\sum_{n=0}^\infty f_n q^{2(n+r)}z^{n+r} + \sum_{n=0}^{\infty}b_nz^n 	\sum_{n=0}^\infty f_n q^{n+r}z^{n+r} + \sum_{n=0}^\infty d_n z^n \sum_{n=0}^\infty f_n z^{n+r} =0.
\end{align}
Then the coefficient of $z^{n+r}$ equals to 0, i.e.
\begin{align}
f_n(q^{2(n+r)}+q^{n+r}b_0(q)+d_0(q)) + \sum_{k=1}^{n-1}f_k(q^{k+r}b_{n-k}+d_{n-k})=0	,
\end{align}
with initial condition
\begin{align}
f_0(q^{2r}+q^r b_0(q) + d_0(q))=0,	
\end{align}
then we obtain a necessary condition for non-zero solution, i.e.
\begin{align}
q^{2r}+q^r b_0(q) + d_0(q) = 0,	
\end{align}
we call
\begin{align}
x^2+b_0(q)x+d_0(q)=0, \label{char-equ-1}	
\end{align}
the \emph{characteristic equation}. Furthermore, if $q^{2(n+r)}+q^{n+r}b_0(q)+d_0(q) \neq  0 $, i.e. $q^{n+r}$ is not a solution of characteristic equation for $n \geq 1$, then
\begin{align}
f_n = \frac{-1}{q^{2(n+r)}+q^{n+r}b_0(q)+d_0(q)}	 \left( \sum_{k=1}^{n-1}f_k(q^{k+r}b_{n-k}+d_{n-k})\right)
\end{align}

{\bf{Conclusion:}} Suppose $q^{r_1}, q^{r_2}$ are the $q^r$-solutions of characteristic equation (\ref{char-equ-1}) , 
\begin{itemize}
 \item{Case 1:} If $r_1 - r_2 \notin \mathbb{Z} $, then there exist two solutions of the form
\begin{align}
 z^{r_1}F_1 , \quad z^{r_2}F_2,	
\end{align}
for $F_1$ and $F_2$ are all power series. 

\item{Case 2:} If $r_1-r_2=n \in \mathbb{Z}_{+}$, then there exists a solution of the form
\begin{align}
z^{r_1}F_1,	
\end{align}
where $F_1$ is a power series.
\end{itemize}
\end{Example}

In order to construct the second solution in case 2, we have to use the following single-value function satisfying 
\begin{align}
\sigma_q f = f+1,	
\end{align}
i.e.
\begin{align}
\ell_q(z) := z \frac{\theta^{\prime}_q(z)}{\theta_q(z)},	
\end{align}
the $q$-logarithm defined in Definition \ref{q-log}. And we will require an intermediary result on it.
\begin{Lemma}[\cite{Adams} and \cite{2019arXiv191100254R}, Lemma VI.1.1.10]
	Let $N \in \mathbb{Z}_{\geq 0}$. The family consisting of the functions $\ell_q(z)^i \in \mathcal{M}(\mathbb{C}^*)$ for $ i \in \{0, \ldots, N \} $ is linearly independent over field $\mathcal{M}(\mathbb{E}_q)$.
\end{Lemma}

We use the following a little bit more general example to explain this idea. 

\begin{Example}
Let's come back to degree two difference equation
\begin{align}
	\left[ \sigma_q^2 + b(z,q) \sigma_q +d(z,q) \right]f=0, \label{degree-2-diiference-equation}
\end{align}
then the characteristic equation is 
\begin{align}
 x^2 +b(0,q) x +d(0,q) = 0.	
\end{align}
Suppose there are two roots (not necessary $q^r$--roots)
\begin{align}
	c_1, \quad c_2.
\end{align}

{\bf{Case 1}}
if  $c_1/c_2 \notin q^{\mathbb{Z}} $, then there are two solution of the form
\begin{align}
	e_{q,c_1} \sum_{n=0}^\infty f_n z^n , \quad e_{q,c_2} \sum_{n=0}^{\infty} g_n z^n.
\end{align}
The computation is similar as Example \ref{Example-1}.
\vspace{0.3cm}

{\bf{Case 2}}
if $c_1=c_2 q^{n_0}$, $n_0 \in \mathbb{Z}_+ $, then the first solution is of the form
\begin{align}
	e_{q,c_1} \sum_{n=0}^{\infty}f_n z^n.
\end{align}
Now let's construct the second solution of the following form
\begin{align}
\ell_q(z) e_{q,c_1} \sum_{n=0}^{\infty} f_n z^n + e_{q,c_2} \sum_{n=0}^{\infty}g_n z^n,	
\end{align}
substituting into (\ref{degree-2-diiference-equation}), we have
\begin{align}
	0&= (\ell_q(z) + 2 ) c_1^2 \cdot e_{q,c_1} \sum_{n=0}^{\infty} f_n q^{2n} z^n + (\ell_q(z)+1)c_1 \cdot e_{q,c_1} \sum_{n=0}^{\infty}f_n q^n z^n \sum_{n=0}^{\infty}b_n(q)z^n  \\
	&+\ell_q(z) e_{q,c_1} \sum_{n=0}^{\infty}f_n z^n \sum_{n=0}^{\infty}d_n(q)z^n   \\
	&+ c_2^2 \cdot e_{q,c_2} \sum_{n=0}^{\infty} g_n q^{2n}z^n + c_2 \cdot e_{q,c_2}\sum_{n=0}^{\infty}g_n q^n z^n \sum_{n=0}^{\infty}b_n(q) z^n \\ 
	&+ e_{q,c_2}\sum_{n=0}^{\infty}g_n q^n z^n \sum_{n=0}^{\infty}d_n(q) z^n.
\end{align}

\begin{itemize}
\item $\ell_q(z)$-term:	
\begin{align}
	0&=\ell_q(z)\cdot \left[ c_1^2 \cdot e_{q,c_1} \sum_{n=0}^{\infty} f_nq^{2n}z^n + c_1\cdot e_{q,c_1} \sum_{n=0}^{\infty} f_nq^nz^n \sum_{n=0}^{\infty} b_nz^n +e_{q,c_1} \sum_{n=0}^{\infty} f_nz^n \sum_{n=0}^{\infty}d_nz^n \right] \\
	&=\ell_q(z) \cdot e_{q,c_1} \cdot \sum_{n=0}^{\infty} \left[ (c_1^2q^{2n}+c_1b_0q^n+d_0)f_n + \sum_{k=0}^{n-1}\left( q^kb_{n-k} +d_{n-k} \right)f_k \right] z^n,
\end{align}
with initial condition 
\begin{align}
	(c_1^2+b_0c_1+d_0)f_0=0,
\end{align}
which is automatically satisfied, so $f_0$ is arbitrary.

\item The remaining term:
\begin{align}
0 &= 2c_1^2 \cdot e_{q,c_1} \sum_{n=0}^{\infty}f_nq^{2n}z^n + c_1 \cdot e_{q,c_1}\sum_{n=0}^\infty f_n q^n z^n \cdot \sum_{n=0}^{\infty}b_n(q) z^n	 \\
&+ c_2^2 \cdot e_{q,c_2} \sum_{n=0}^{\infty} g_nq^{2n}z^n + c_2 \cdot e_{q,c_2} \sum_{n=0}^{\infty} g_nq^nz^n\cdot \sum_{n=0}^{\infty}b_n(q)z^n   \\
&+e_{q,c_2}\sum_{n=0}^{\infty}g_n z^n \sum_{n=0}^{\infty} d_n(q)z^n. \label{remaining-term}
\end{align}
Since $c_1=c_2 q^{n_0}$, then
\begin{align}
	e_{q,c_1} = q^{\frac{n_0(n_0-1)}{2}}z^{n_0}e_{q,c_2}
\end{align}
so (\ref{remaining-term}) becomes
\begin{align}
0 &= 2c_1^2 \cdot q^{\frac{n_0(n_0-1)}{2}}z^{n_0}e_{q,c_2} \sum_{n=0}^{\infty}f_nq^{2n}z^n + c_1 \cdot q^{\frac{n_0(n_0-1)}{2}}z^{n_0}e_{q,c_2}\sum_{n=0}^\infty f_n q^n z^n \cdot \sum_{n=0}^{\infty}b_n(q) z^n	 \\
&+ c_2^2 \cdot e_{q,c_2} \sum_{n=0}^{\infty} g_nq^{2n}z^n + c_2 \cdot e_{q,c_2} \sum_{n=0}^{\infty} g_nq^nz^n\cdot \sum_{n=0}^{\infty}b_n(q)z^n   \\
&+e_{q,c_2}\sum_{n=0}^{\infty}g_n z^n \sum_{n=0}^{\infty} d_n(q)z^n	  \\
&=e_{q,c_2} \left\{ \sum_{n=0}^{\infty} \left[\left(2c_1^2\cdot q^{\frac{n_0(n_0-1)}{2}}q^{2n} + b_0c_1q^{\frac{n_0(n_0-1)}{2}}q^{n}  \right)f_n + \sum_{k=0}^{n-1} q^{\frac{n_0(n_0-1)}{2}}c_1f_kq^{k}b_{n-k}  \right]z^{n+n_0}   \right.  \\  
&+ \left. \sum_{n=0}^{\infty} z^n \left[ \left( c_2^2 q^{2n} + c_2 q^nb_0+d_0 \right)g_n + \sum_{k=0}^{n-1} \left(c_2q^kb_{n-k}  + d_{n-k} \right)g_k  \right]   \right\}. \label{n>n0}
\end{align}
\end{itemize}
Thus for $n < n_0$, $(c_2q^n)^2+b_0c_2q^n+d_0 \neq 0 $, we obtain
\begin{align}
	g_n = \frac{-1}{(c_2q^n)^2+b_0c_2q^n+d_0} \left[ \sum_{k=0}^{n-1} \left( c_2q^kb_{n-k} + d_{n-k} \right)g_k  \right].
\end{align}
For $n > n_0$, from (\ref{n>n0}) one could see that $g_n$ is determined by $\{f_0,\ldots, f_{n-n_0} \}$. \\

\noindent For $n=n_0$, from (\ref{n>n0}) we have
\begin{align}
0=q^{\frac{n_0(n_0-1)}{2}}(2c^2_1+b_0c_1)f_0 + \sum_{k=0}^{n_0-1}(c_2q^kb_{n-k}+d_{n-k})g_k, \label{f0}
\end{align}
since $c_1,c_2$ are two roots of
\begin{align}
	x^2 +b_0 x +d_0 =0,
\end{align}
then 
\begin{align}
b_0 = -(c_1+c_2),  \ d_0 = c_1c_2,	
\end{align}
so
\begin{align}
2c_1+b_0c_1=2c_1^2 - (c_1+c_2)c_1 = c_1^2 -c_1c_2.	
\end{align}
If $c_1 \neq c_2$, i.e. $n_0 \neq 0$, then from (\ref{f0}) we have
\begin{align}
	f_0 = -\frac{q^{\frac{n_0(n_0-1)}{2}}}{c_1(c_1-c_2)} \sum_{k=0}^{n_0-1}(c_2q^kb_{n-k}+d_{n-k})g_k,
\end{align}
so $f_0$ is determined by $g_0$. Thus we have two free parameters $g_0$ and $g_{n_0}$.  
\vspace{0.3cm}

{\bf{Case 3}} if $c=c_1=c_2$, i.e. double roots
\begin{align}
   0=x^2+b_0(q)x+d_0(q)=(x-c)^2,
\end{align}
so $b_0=-2c$ and $d_0=c^2$, now consider the solution of the form
\begin{align}
	\ell_q(z) e_{q,c} \sum_{n=0}^{\infty}f_nz^n + e_{q,c}\sum_{n=0}^{\infty}g_nz^n.
\end{align}                                                                                                                                                                                                                                                                                                                                                                                                                                                                                                                                                                                                                                                                                                                                                                                                                                                                                                                                                                                                                                                                                                                                                                                                                                                                                                                                                                                                                                                                                                                                                                                                                                                                                                                                                                                                                                                                                                                                                                                                                                                                                                                                                                                                                                                                                           
substituting into (\ref{degree-2-diiference-equation}), we have
\begin{align}
	&(\ell_q(z)+2)c^2e_{q,c}\sum_{n=0}^{\infty}f_n(q)q^{2n}z^n + (\ell_q(z)+1)c \cdot e_{q,c}\sum_{n=0}^{\infty}f_n(q)q^nz^n\sum_{n=0}^{\infty}b_n(q)z^n \\
	+& \ell_q(z)e_{q,c}\sum_{n=0}^{\infty}f_n(q)z^n\sum_{n=0}^{\infty}d_n(q)z^n \\
	+& c^2e_{q,c}\sum_{n=0}^{\infty}g_n(q)q^{2n}z^n +c \cdot e_{q,c} \sum_{n=0}^{\infty}g_n(q)q^nz^n\sum_{n=0}^{\infty}b_n(q)z^n+e_{q,c}\sum_{n=0}^{\infty}g_n(q)z^n\sum_{n=0}^{\infty}d_n(q)z^n.
\end{align}
\begin{itemize}
\item $\ell_q(z)$-term: 
\begin{align}
0&=\ell_q(z)  e_{q,c} \cdot \sum_{n=0}^{\infty} \left[ \left(q^{2n}c^2+cq^nb_0+d_0 \right)f_n + \sum_{k=0}^{n-1}\left(cq^kb_{n-k}+d_{n-k} \right)f_k \right]z^n,  \label{fn}	
\end{align}
with initial condition 
\begin{align}
(c^2+b_0c+d_0)f_0=0	
\end{align}
so $f_0$ is arbitrary. Since $(cq^n)^2+b_0(cq^n)+d_0 \neq 0$ for $n \geq 1 $, from (\ref{fn}) we have
\begin{align}
	f_n = \frac{-1}{q^{2n}c^2+cq^nb_0+d_0}\sum_{k=0}^{n-1}\left(cq^kb_{n-k}+d_{n-k} \right)f_k.
\end{align}	

\item The remaining term:
\begin{align}
0 &= e_{q,c}\sum_{n=0}^{\infty} \left[\left(2c^2q^{2n} + cq^nb_0  \right)f_n + \sum_{k=0}^{n-1}cq^kb_{n-k}f_k \right. \\
  & + \left.  \left(c^2q^{2n}+cq^nb_0+d_0 \right)g_n + \sum_{k=0}^{n-1} \left(rq^kb_{n-k} +d_{n-k}  \right)g_{k}   \right]z^n	
\end{align}
with initial condition 
\begin{align}
2r^2+b_0r=0
\end{align}
so $g_0$ is arbitrary. Then $g_n$ is determined by
\begin{align}
\{g_0, \ldots,g_{n-1} \} \ and \  \{f_0,\ldots,f_n \}.	
\end{align}
\end{itemize}
{\bf{Conclusion:}} 
\begin{itemize}
\item{Case 1}: if  $c_1/c_2 \notin q^{\mathbb{Z}} $, then there are two solution of the form
\begin{align}
	e_{q,c_1} \sum_{n=0}^\infty f_n z^n , \quad e_{q,c_2} \sum_{n=0}^{\infty} g_n z^n.
\end{align}

\item{Case 2}: if $c_1=c_2 q^{n_0}$, $n_0 \in \mathbb{Z}_{\geq 0} $, there are two solutions of the form
\begin{align}
	e_{q,c_1} \sum_{n=0}^{\infty}f_n z^n, \quad \ell_q(z) e_{q,c_1} \sum_{n=0}^{\infty} f_n z^n + e_{q,c_2} \sum_{n=0}^{\infty}g_n z^n,	
\end{align}
\end{itemize}
\end{Example}

\subsection{The $q$-hypergeometric equation}
Consider following degree 2 difference equation
\begin{align}
\left[ (1-\sigma_q)(1-q^r \sigma_q ) - z(1-q^\alpha \sigma_q)(1-q^\beta \sigma_q) \right] f =0.
\end{align}
Suppose $r \notin \mathbb{Z} $, then the characteristic equation is
\begin{align}
(1-x)(1-q^r x)=0,	
\end{align}
then we have two $q^{\mathbb{Z}}$-roots
\begin{align}
x=q^0 \ or \	x=q^{-r}.
\end{align}

For the first root $x=q^0$, let
\begin{align}
f=\sum_{n=0}^{\infty}f_n z^n	,
\end{align}
then 
\begin{align}
\sum_{n=0}^{\infty}f_n(1-q^n)(1-q^{n+r})z^n - \sum_{n=0}^{\infty}f_n(1-q^{n+\alpha})(1-q^{n+\beta})z^{n+1}=0	
\end{align}
thus we obtain
\begin{align}
f_n(1-q^n)(1-q^{n+r}) = f_{n-1}(1-q^{n-1+\alpha})(1-q^{n-1+\beta}), \label{hyper-relation}	
\end{align}
with initial condition
\begin{align}
f_0(1-q^0)(1-q^r)=0,	
\end{align}
for arbitrary $f_0$. From (\ref{hyper-relation}), we obtain
\begin{align}
\frac{f_n}{f_{n-1}} = \frac{(1-q^{\alpha+n-1})(1-q^{\beta+n-1})}{(1-q^n)(1-q^{r+n})}	,
\end{align}
so
\begin{align}
f_n &= \frac{f_n}{f_{n-1}} \cdots \frac{f_1}{f_0}= \frac{(1-q^\alpha)\cdots(1-q^{\alpha+n-1})(1-q^\beta)\cdots(1-q^{\beta+n-1}) }{(1-q)\cdots(1-q^n)(1-q^{r+1})\cdots(1-q^{r+n})}	 \\
    &= \frac{(q^\alpha;q)_n(q^\beta;q)_n}{(q;q)_n(q^{r+1};q)_n}.
\end{align}
Here we use the $q$-Pochhammer symbol
\begin{align}
(a;q)_n := (1-a)\cdots (1-aq^{n-1}).	
\end{align}
So we obtain first solution
\begin{align}
	F_1 = \sum_{n=0}^{\infty} \frac{(q^\alpha;q)_n(q^\beta;q)_n}{(q;q)_n(q^{r+1};q)_n} z^n,
\end{align}
this is $q$-hypergeometric series.

For another root $x=q^{-r}$, similarly, we obtain a solution as follows
\begin{align}
F_2 = \sum_{n=0}^{\infty}\frac{(q^{\alpha-r};q)_n(q^{\beta-r};q)_n}{(q;q)_n(q^{1-r};q)_n}	z^n.
\end{align}

\begin{Remark}
Generalized $q$-hypergeometric series is as follows
\begin{align}
{}_r\phi_s(a_1,\ldots,a_r;b_1,\ldots,b_s;q;x )=\sum_{n=0}^{\infty} \frac{(a_1;q)_n\cdots(a_r;q)_n}{(b_1;q)_n \cdots (b_s;q)_n}. ((-1)^nq^{\frac{n(n-1)}{2}})^{1+s-r} z^n	
\end{align}
It satisfies the following difference equation
\begin{align}
	\left[ \prod_{i=1}^s (1-b_i\sigma_q) - z \left(-\sigma_q \right)^{1+s-r} \prod_{i=1}^r (1-a_i\sigma_q)  \right] G = 0.
\end{align}	
This type of difference equations often appear in quantum K-theory, especially, $I$-function with level structures, see \cite{RZ} for more details.
\end{Remark}

\subsection{Difference equation for $\mathbb{P}^1$ with standard level structure ($0 \leq l \leq 2$)} \label{sub-proj}
As we mentioned in the introduction, the modified $I$-function of $\mathbb{P}^1$ with level structure is
\begin{align} 
\widetilde{I^{K, l}_{\mathbb{P}^1}} = 	P^{\ell_q(z)}\sum_{d=0}^{\infty} \frac{\left(P^{d}q^{\frac{d(d-1)}{2}} \right)^l z^d}{\prod_{k=1}^{d} (1-Pq^k)^2 },
\end{align}
satisfying the following difference equation
\begin{align}
\left[ (1-\sigma_q)^2 -z  \sigma_q^l  \right] \widetilde{I^{K,St}_{\mathbb{P}^1}} = 0.	
\end{align}
If $0 \leq l \leq 2$, then the above difference equation is regular singular. 

The characteristic equation is
\begin{align}
(1-x)^2 = 0,	
\end{align}
then we have double roots 
\begin{align}
x_1=x_2=q^0.	
\end{align}
From (\ref{sol-resonant-0})-(\ref{sol-resonant-ki}) we know there exist two solutions of the following form
\begin{align}
 F_1(z,q), \quad \ell_q(z) F_1(z,q) + F_2(z,q),	
\end{align}
where
\begin{align}
F_i(z,q) = \sum_{k=0}^{\infty} f_{ik}(q)z^k, {\rm with} \ f_{i0}=1, \ i=1,2.
\end{align}
As before, one can easily find that
\begin{align}
F_1(z,q) = \sum_{d=0}^{\infty} \frac{\left(q^{\frac{d(d-1)}{2}} \right)^l z^d}{\prod_{k=1}^{d} (1-q^k)^2 }.	
\end{align}
Substituting the second solution into the difference equation, we have
\begin{align}
\left[  2\sigma_q(\sigma_q -1) - lz\sigma_q^l \right] F_1(z,q) =\left[z\sigma_q^l-(1-\sigma_q)^2 \right] F_2(z,q).	
\end{align}
Then we obtain the following recursive formula for $F_2(z,q)$:
\begin{align}
f_{2,d}(q) = \frac{q^{l(d-1)}}{(1-q^d)^2} f_{2,d-1}(q) + \frac{q^{ld(d-1)/2}}{\prod_{k=1}^d(1-q^k)^2}	\left( \frac{2q^d}{1-q^d} + l  \right)
\end{align}
So we have
\begin{align}
 & \ell_q(z) F_1(z,q) + F_2(z,q) \nonumber \\
=& \sum_{d=0}^\infty \frac{q^{\frac{ld(d-1)}{2}}}{\prod_{k=1}^d(1-q^k)^2} \left(\ell_q(z) - \sum_{k=1}^d \frac{2q^k}{1-q^k}  \right) + 	\sum_{d=0}^\infty \frac{ld \cdot q^{\frac{ld(d-1)}{2}}}{\prod_{k=1}^d(1-q^k)^2}.
\end{align}

\subsection{Convergence of solutions} \label{reg-convergence}
In this subsection, we prove the convergence of solutions in reuglar singular cases. Here we follow \cite{Adams}.

Let's consider the regular singular equation
\begin{align}
\sum_{k=0}^n a_k(z) (\sigma_q)^k f(z)=0,	
\end{align}
with
\begin{align}
a_k(z)= a_{k0}+a_{k1}z+a_{k2}z^2 + \cdots,	
\end{align}
i.e., the coefficients $a_k(z)$ satisfy conditions in Proposition \ref{critieria-regular}. 
\vspace{0.3cm}

\noindent {\bf Assumptions:}
\begin{itemize}
\item[($\dagger$)] We assume that $a_{00}, a_{n0} \neq 0$, without loss of generality we may assume $a_n(z) \equiv 1 $.
\item[($\dagger\dagger$)] It will be assumed further that all of the power series $a_k(z)$ to be analytic at the origin and have radius of convergence $>1$.	
\end{itemize}

To see the convergence of solutions, it's sufficient to show the following two solutions in resonant case:
\begin{align}
S_{i0}(z,q)&=e_{q,c_i}F_{i0}(z,q), \label{convergence-solution-1} \\
	S_{i1}(z,q)&=\ell_q(z) S_{i0}(z,q)+ e_{q,c_iq^{m_{i1}}}F_{i1}(z,q),	\label{convergence-solution-2}
\end{align}
are convergent. We begin by proving directly the convergence of the series of the first one. 

Suppose
\begin{align}
	F_{i0}(z,q)=  \sum_{d=0}^\infty f_{i,0,d}(q) z^d,
\end{align}
since $\sigma_q e_{q,c_i}(z)=c_i \cdot e_{q,c_i}(z)$, then
\begin{align}
	\sum_{k=0}^n c_i^k \cdot a_k(z)  (\sigma_q)^k F_{i0}(z,q)=0.
\end{align}
i.e.,
\begin{align}
	\sum_{k=0}^n c_i^k \cdot \left(\sum_{j=0}^\infty a_{kj} z^j \right) \left( \sum_{d=0}^\infty q^{kd} \cdot f_{i,0,d} \cdot z^{d}  \right) =0.
\end{align}
Then the coefficient of $z^m$-term is 
\begin{align}
\sum_{j+d=m} \sum_{k=0}^n \left( c_i^k \cdot q^{kd} \cdot a_{kj} \right) f_{i,0,d}	= 0.
\end{align}
In the notation at the beginning of this section, $m_0=m_{i,k_i}+1$ is the first positive integer that 
\begin{align}
a_{00}+ a_{10} (c_iq^{m_0}) + \ldots, a_{n0}(c_iq^{m_0})^n \neq 0.	
\end{align}
By deduction, one could let $f_{i,0,d}=0$ for $0 \leq d < m_0 $, and $f_{i,0,m_0} = 1$. Let
\begin{align}
L_{ijd} = \sum_{k=0}^n \left(  c_i^k \cdot q^{kd} \cdot a_{kj} \right).	
\end{align}
Then the coefficient $f_{i,0,m}$ is determined by the relation
\begin{align}
f_{i,0,m} = -\frac{\sum_{j=1}^m L_{i,j,m-j} \cdot f_{i,0,m-j}}{L_{i,0,m}}, \quad m > m_0.	
\end{align}
Considering first the numerator of this quotient, by the assumption ($\dagger$), $a_n(z) \equiv 1$, then
\begin{align}
|L_{i,0,m}|	= |(c_iq^m)^n| \cdot | 1 + \frac{a_{n-1,0}}{c_iq^m} + \ldots + \frac{a_{00}}{(c_iq^m)^n} |,
\end{align}
let $m=m_1$ be large enough that the second factor on the right is $> \frac{1}{2}$. For $m_0 \leq m< m_1$, we have $|L_{i,0,m}| = A_m |(c_iq^m)^n| $, where $A_m \neq 0$. Setting 
\begin{align}
A = {\rm min} \big\{ A_{m_0}, \ldots, A_{m_1-1}, \frac{1}{2} \big\}.	
\end{align}
And by the assumption ($\dagger\dagger$), $a_{i}(z)$ are convergent for $z=1$, therefore, we have
\begin{align}
|a_{ij}| < M_1.	
\end{align}
The set $|c_i^k|$, for $k=0,1, \ldots, n$ have an upper bound $M_2$, then for $j\geq 1 $, we have
\begin{align}
L_{ijm}& = \sum_{k=0}^n \left( (c_iq^m)^k \cdot a_{kj} \right) \\
       & \leq n M_1M_2 |q|^{m(n-1)}, 
\end{align}
Then we have
\begin{align}
|f_{i,0,m}| < \frac{nM_1M_2|q|^{(m-1)(n-1)}\sum_{j=0}^{m-1}|f_{i,0,j}|}{A|(c_iq^m)^n|} .
\end{align}
Defining $\widetilde{f_{i,0,0}}$ as $|f_{i,0,0}|$, we obtain the following upper bound for $f_{i,0,m}$:
\begin{align}
|f_{i,0,m}| < \frac{nM_1M_2\sum_{j=0}^{m-1}|\widetilde{f}_{ij}|}{A|c_i^n | | q |^{n+m-1}}=: \widetilde{f}_{i,0,m}	.	
\end{align}
As $m$ becomes infinite the limit of
\begin{align}
\frac{\widetilde{f}_{i,0,m}}{\widetilde{f}_{i,0,m+1}} = 	\frac{|q|\sum_{j=0}^{m-1}|{\widetilde{f}_{i,0,j}|}}{\widetilde{f}_{im}+ \sum_{j=0}^{m-1}|{\widetilde{f}_{i,0,j}|} } = \frac{|q|}{nM/(|q|^{n+k-1}+1)},
\end{align}
is $|q|$, where $M = M_1M_2/(A|c_i^n|)$. It implies the convergence of the power-series solution (\ref{convergence-solution-1}).

Let us next consider the second solution (\ref{convergence-solution-2}), i.e.
\begin{align}
S_{i1}(z,q)&=\ell_q(z) e_{q,c_i}F_{i0}(z,q)+ e_{q,c_iq^{m_{i1}}}F_{i1}(z,q),
\end{align}
and 
\begin{align}
e_{q,c_iq^{m_{i1}}}(z) = q^{\frac{m_{i1}(m_{i1}-1)}{2}} z^{m_{i1}} \cdot e_{q,c_i}(z).	
\end{align}
Setting $q_0 = q^{\frac{m_{i1}(m_{i1}-1)}{2}}$, then from the difference equation, we obtain
\begin{align}
& \sum_{k=0}^n \left( \sum_{j=0}^\infty a_{kj} z^j \right) k c^k_i \left(\sum_{d=0}^\infty q^{kd} f_{i,0,d} z^d    \right)	  \\
+& q_0 \sum_{k=0}^n \left( \sum_{j=0}^\infty a_{kj} z^j \right) c_i^k \left(\sum_{d=0}^\infty q^{k(d+m_{i1})} f_{i,1,d} z^{d+m_{i1}}    \right) = 0.
\end{align}
Setting
\begin{align}
L^{\prime}_{ijd}        &= \sum_{k=0}^n \left( k c_i^k \cdot q^{kd} \cdot a_{kj} \right),
\end{align}
we find that the $f_{i,1,m}$ satisfy the relations
\begin{align}
f_{i,1,m-m_{i1}} = -\frac{L^{\prime}_{i,0,m} \cdot f_{i,0,m}+ \sum_{j=1}^{m-1}\left[L^{\prime}_{i,j,m-j} \cdot f_{i,0,m-j} + q_0 L_{i,j,m-j} \cdot f_{i,1,m-j-m_{i1}}  \right]}{q_0 L_{i,0,m}}.	
\end{align}
Here we use the notation that $f_{i,1,m}=0$ for $m < 0$. By deduction, one could let $f_{i,1,d}=0$ for $0 \leq d < m_0-m_{i1} $, and $f_{i,1,m_0-m_{i1}} = 1$. Proceeding as before, we find
\begin{align}
|f_{i,1,m-m_{i1}}| &< \frac{n M \sum_{j=1}^{m-1}\left| f_{i,1,j-m_{i1}} \right|}{ |q|^{n+m-1}}+\frac{n(n-1) M \sum_{j=1}^{m-1}\left|f_{i,0,j}\right|}{|q|^{n+k-1}}+\frac{n(n+1) M\left|f_{i,0,m}\right|}{2}	, \\
                   &< \frac{n^2(n+1)M^2\sum_{j=0}^{m-1}\left[ |f_{i,1,j-m_{i1}}| + 2|f_{i,0,j}| \right]}{|q|^{n+m-1}}.
\end{align}
To obtain the last relation, we use the inequality for $f_{i,0,m}$. Similarly, defining $\widetilde{f}_{i,1,0}$ as $|{f}_{i,1,0}|$ and $\widetilde{f}_{i,1,m}$ as before, then we have
\begin{align}
\frac{\widetilde{f}_{i,1,m+1-m_{i1}}}{\widetilde{f}_{i,1,m-m_{i1}}} = \frac{1}{|q|} \cdot  \left[ \frac{n^2(n+1)M^2}{|q|^{n+m-1}} +1 + \frac{2nM}{|q|^{n+m-1}} \frac{\sum_{j=0}^{m-1}\widetilde{f}_{i,0,j}}{\sum_{j=0}^{m-1}\left(\widetilde{f}_{i,1,j-m_{i1}}+2\widetilde{f}_{i,0,j} \right)} \right]	.
\end{align}
The series $\sum_{j=0}^\infty \widetilde{f}_{i,0,j}$ is convergent. As $m$ tends to be infinite, $\sum_{i=0}^{m-1}\widetilde{f}_{i,1,j-m_{i1}}$ must either approach a limit or become infinite, in either case the limit of the above quotient is $1/|q|$. Hence it proofs the convergence of the second solution.

\section{Local solutions for irregular singualr cases}
In the last section we show how to construct solutions in regular singular case. In this section we focus on irregular case, and we follow the method of \cite{Adams2} and show how to construct formal series solutions in some irregular singular degree 2 difference equations.

Let's consider the following degree 2 difference equation
\begin{align}
	\left[a_2(z,q)\sigma _q^2 + a_1(z,q)\sigma_q + a_0(z,q)  \right]f=0,
\end{align}
with
\begin{align}
	a_i(z,q) = a_{i0}+a_{i1}z + a_{i2}z^2 + \cdots, \ i=0,1,2,
\end{align}
for one of $a_{i0} \neq 0$ and whose characteristic equation is
\begin{align}
a_{20}z^2 + a_{10}z+a_{00}=0	.
\end{align}
From Proposition \ref{critieria-regular}, we know that irregularity implies that one or both of $a_{20}, a_{00}=0$.

\begin{Example}[Ramanujan Equation]
Let's consider the following Ramanujan equation
\begin{align}
	\left[ qz \sigma_q^2 - \sigma_q +1 \right]f=0 \label{Ramanujan-eqn},
\end{align}	
whose characteristic equation is
\begin{align}
	-x+1=0 \  i.e. \ x=1=q^0.
\end{align}
Consider a solution of the form
\begin{align}
	F=\sum_{n=0}^{\infty} f_n(q)z^n,
\end{align}
substituting into (\ref{Ramanujan-eqn}), we have
\begin{align}
	0 &=qz\sum_{n=0}^{\infty}f_nq^{2n}z^n - \sum_{n=0}^{\infty}f_nq^nz^n + \sum_{n=0}^{\infty}f_nz^n \\
	  &=\sum_{n=1}^{\infty}f_{n-1}q^{2n-1}z^n - \sum_{n=0}^{\infty}f_nq^nz^n + \sum_{n=0}^{\infty} f_nz^n.
\end{align} 
Therefore, we have
\begin{align}
	f_{n-1} q^{2n-1} - (q^n-1)f_n = 0,
\end{align}
with initial condition $(1-1)f_0=0$, so $f_0$ is a free parameter. Then
\begin{align}
	\frac{f_n}{f_{n-1}} = - \frac{q^{2n-1}}{1-q^n} = - \frac{q^{2(n-1)+1}}{1-q^n},
\end{align}
i.e.
\begin{align}
	f_n &= \frac{f_n}{f_{n-1}}\cdot\frac{f_{n-1}}{f_{n-2}}\cdots \frac{f_1}{f_0} \\
	    &=\frac{-q^{2n-1}}{1-q^n} \cdot \frac{-q^{2n-3}}{1-q^{n-1}} \cdots \frac{-q}{1-q} \\
	    &= \frac{(-1)^{n}(q^{\frac{n(n-1)}{2}})^2q^n}{(q;q)_q}.
\end{align}
Then the first solution is of the form
\begin{align}
	F=\sum_{n=0}^{\infty}\frac{(q^{\frac{n(n-1)}{2}})^2}{(q;q)_n}(-qz)^n = {}_0 \phi_1 (-,0;q;-qx),
\end{align}
where 
\begin{align}
	{}_r\phi_s(a_1,\ldots,a_r;b_1,\ldots,b_s;q;x ) := \sum_{n=0}^{\infty} \frac{(a_1;q)_n\cdots(a_r;q)_n}{(b_1;q)_n\cdots (b_s;q)_n}((-1)^nq^{\frac{n(n-1)}{2}})^{1+s-r}z^n.
\end{align}
Let's consider another solution of the form
\begin{align}
	\theta^{-1}_q(z) \sum_{n=0}^{\infty} f_n(q)z^n,
\end{align}
substituting into (\ref{Ramanujan-eqn}), we obtain
\begin{align}
	0&=qz \cdot q^{-1}z^{-2} \theta^{-1}_q(z) \sum_{n=0}^{\infty}f_n(q)q^{2n}z^n- z^{-1}\theta^{-1}_q(z)\sum_{n=0}^{\infty}f_nq^nz^n + \theta^{-1}_q(z)\sum_{n=0}^{\infty}f_nz^n \\
	 &=\theta^{-1}_q(z)\sum_{n=0}^{\infty}f_n(q)q^{2n}z^{n-1} - \theta^{-1}_q(z) \sum_{n=0}^{\infty}f_nq^nz^{n-1} + \theta^{-1}_q(z) \sum_{n=0}^{\infty}f_nz^n \\
	 &=\theta^{-1}_q(z)\sum_{n=0}^{\infty} \left( f_{n+1}(q)q^{2(n+1)} - f_{n+1}(q)q^{n+1}+f_n \right)z^n,
\end{align}
with initial condition
\begin{align}
f_0(q^0-q^0)=0.	
\end{align}
For $n > 0 $, we have
\begin{align}
	f_{n+1}(q^{2(n+1)-q^{n+1}}-q^{n+1})+f_n = 0,
\end{align}
then
\begin{align}
	\frac{f_{n+1}}{f_{n}} = \frac{1}{q^{n+1}-q^{2(n+1)}} = \frac{1}{q^{n+1}(1-q^{n+1})},
\end{align}
so
\begin{align}
f_n = f_0 \frac{1}{q^n(1-q^n)} \cdots \frac{1}{q(1-q)} = \frac{(q^{\frac{n(n-1)}{2}})^{-1}}{(q;q)_n} \frac{1}{q^n}.	
\end{align}
To summarize what has been mentioned above, we obtain the second solution
\begin{align}
F_2 = \theta^{-1}_q(z) \cdot {}_2\phi_0 (0,0;-;q;-\frac{x}{q})	
\end{align}
\end{Example}
Here we use the Jacobi's theta function introduced in section 2.2. This is not a coincidence, in the following we show how to do in general.

\subsection{General Technique: Newton polygon}
Let's consider the equation
\begin{align}
\sum_{i=0}^n a_i(z) (\sigma_q)^i f(z)=0,	
\end{align}
with
\begin{align}
a_i(z)= a_{i0}+a_{i1}z+a_{i2}z^2 + \cdots.	
\end{align}
Denoted by $a_{i,j_i}$ the first nonzero coefficient in $a_i(z)$, and choosing $i-$ and $j-$axes as horizontal and vertical axes respectively, plot the points $(n-i, j_i)$. Construct a broken line, convex downward, such that both ends of each segment of the line are points of the set $(n-i, j_i)$. Then we obtain a Newton polygon as follows
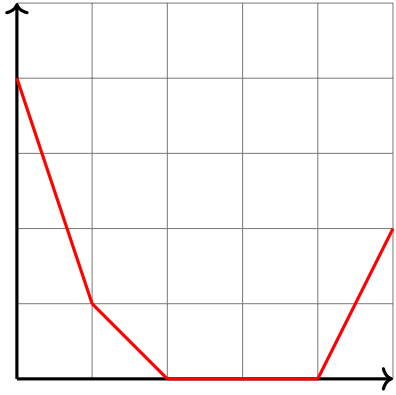
\begin{center}
\begin{tikzpicture}
\draw [xstep=1, ystep=1, draw=gray](0,0)grid(5,5);
\draw [->, very thick](0,0)--(5,0);
\draw [->, very thick](0,0)--(0,5);

\draw [very thick, draw=red](0,4)--(1,1); 
\draw [very thick, draw=red](1,1)--(2,0);
\draw [very thick, draw=red](2,0)--(4,0);
\draw [very thick, draw=red](4,0)--(5,2); 
\end{tikzpicture}
\end{center}
Note that the horizontal segment corresponds to the characteristic equation	
\begin{align}
a_{k,0}x^k+ a_{k-1,0} x^{k-1} + \cdots + a_{d,0}x^d = 0. 
\end{align}
The degree of the above characteristic equation is 1 less than the number of points that lie on or above that segment.

Then the general technique to construct solutions is as follows
\begin{itemize}
\item{\bf{Horizontal segment:}} As mention above, it corresponds to characteristic equation.  Using the non-zero roots, we could construct the associated solutions as regular singular cases.
\item{\bf{Non-horizontal segment:}} For each non-horizontal segment of slope $\mu$, a rational number.
\begin{itemize}
\item If $\mu = r$ is an integer, we consider a formal series solution of the form
\begin{align}
	\theta^{r}_q(z)\sum_{n=0}^{\infty} f_n(q)z^n.
\end{align}

\item If $\mu = t/s$ is a rational number with $s$ positive, then we consider a formal series solution of the form
\begin{align}
   \theta_q^{t/s}(z) \sum_{n=0}^{\infty} f_{n}(q) z^{n/s}.	
\end{align}
\end{itemize}
 
\end{itemize}

\begin{Example}[Slope $\mu= -1$] \label{slope=-1}
Consider the following difference equation
\begin{align}
	\left[ z^2\sigma_q^2 + z \sigma_q +1 \right]f=0.
\end{align}	
The Newton polygon is 
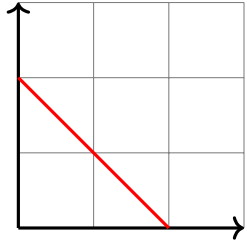
\begin{center}
\begin{tikzpicture}
\draw [xstep=1, ystep=1, draw=gray](0,0)grid(3,3);
\draw [->, very thick](0,0)--(3,0);
\draw [->, very thick](0,0)--(0,3);

\draw [very thick, draw=red](0,2)--(2,0); 
\end{tikzpicture}
\end{center}
with slope $\mu=-1$, then consider the solution of the form
\begin{align}
	\theta^{-1}_q F(z).
\end{align}
Then we obtain a new difference equation for $F(z)$ as follows
\begin{align}
	\left[ q^{-1}\sigma_q^2 + \sigma_q + 1 \right] F(z) =0,
\end{align}
with characteristic equation
\begin{align}
	q^{-1}x^2+x+1=0.
\end{align}
Suppose the roots are $c_1$ and $c_2$, so new difference equation could be rewritten as
\begin{align}
	( \sigma_q - c_1 )(\sigma_q-c_2)F=0,
\end{align}
so it's easy to construct two solutions
\begin{align}
	\theta^{-1}_q(z)e_{q,c_1}(z), \quad \theta^{-1}_q(z)e_{q,c_2}(z).
\end{align}
\end{Example}

\begin{Example}[Slope $\mu=-2$] \label{slope=-2}
Consider the following difference equation
\begin{align}
	\left[z^2\sigma_q^2 - \sigma_q +1  \right]f=0. \label{April-27-diff-eqn}
\end{align}
The Newton polygon of (\ref{April-27-diff-eqn}) is as follows
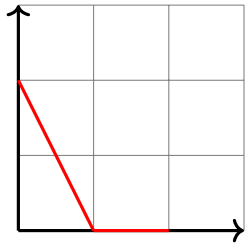
\begin{center}
\begin{tikzpicture}
\draw [xstep=1, ystep=1, draw=gray](0,0)grid(3,3);
\draw [->, very thick](0,0)--(3,0);
\draw [->, very thick](0,0)--(0,3);

\draw [very thick, draw=red](0,2)--(1,0); 
\draw [very thick, draw=red](1,0)--(2,0); 
\end{tikzpicture}
\end{center}
The first segment's slope is $-2$, and the second horizontal segment corresponds to the characteristic equation
\begin{align}
	-x+1=0,
\end{align}
then consider solution of the following form
\begin{align}
	F_1 = \sum^{\infty}_{n=0} f_n z^n,
\end{align}
substituting into (\ref{April-27-diff-eqn}), we obtain
\begin{align}
	0 &= \sum_{n=0}^{\infty}f_n q^{2n} z^{n+2} - \sum_{n=0}^{\infty}f_n q^n z^n + \sum_{n=0}^{\infty} f_n z^n \\
	  &= \sum_{n=2}^{\infty}f_{n-2}q^{2n-4}z^n + \sum^{\infty}_{n=0} f_{n}(1-q^n)z^n, \label{fn-2}
\end{align}
with initial conditions
\begin{align}
	& f_0(1-q^0)=0, \\
	& f_1(1-q^1)=0,
\end{align}
i.e. $f_1=0$ and $f_0$ could be arbitrary. For $n \geq 1$, from (\ref{fn-2}) we have
\begin{align}
	\frac{f_{2n}}{f_{2n-2}}= \frac{-q^{4(n-1)}}{1-q^{2n}}.
\end{align}
Thus
\begin{align}
	f_{2n} &= \frac{-q^{4(n-1)}}{1-q^{2n}} \cdot \frac{-q^{4(n-2)}}{1-q^{2(n-1)}} \cdots \frac{-1}{1-q^2} \cdot f_0 \\
	       &=\frac{(-1)^n\left( q^{\frac{n(n-1)}{2}} \right)^4}{(q^2;q^2)_n}\cdot f_0,
\end{align}
taking $f_0=1$, we finally obtain
\begin{align}
	F_1(z) = \sum^{\infty}_{n=0} \frac{\left( q^{\frac{n(n-1)}{2}} \right)^4}{(q^2;q^2)_n} (-1)^n z^{2n}.
\end{align}

Now let's construct the second solution, consider solution of the following form
\begin{align}
F_2=\theta^{-2}_q(z)F(z),	
\end{align}
substituting into (\ref{April-27-diff-eqn}), we obtain
\begin{align}
	0 &= z^2\sigma_q^2(\theta_q^{-2}(z)F(z)) - \sigma_q(\theta_q^{-2}(z)F(z)) + \theta^{-2}_q(z)F(z)    \\
	  &=z^2q^{-2}z^{-4}\theta_q^{-2}(z) \cdot \sigma_q^2 F(z) -z^{-2}\theta_q^{-2}(z) \sigma_q F(z) + \theta_q^{-2}(z) F(z)    \\
	  &= z^{-2} \theta_q^{-2}(z) \cdot \left[ q^{-2}\sigma_q^2 - \sigma_q +z^2  \right] F(z).
\end{align}
So we obtain a new difference equation 
\begin{align}
	\left[ q^{-2}\sigma_q^2 - \sigma_q +z^2  \right] F = 0, \label{new-eqn}
\end{align}
with characteristic equation 
\begin{align}
	q^{-2}x^2 - x = 0.
\end{align}
The roots are $x=0$ or $x=q^2$. For $x=q^2$, consider the following solution
\begin{align}
	F = \sum_{n=0}^{\infty} f_n z^{n+2}, \label{April-27-solution}
\end{align}
substituting (\ref{April-27-solution}) into (\ref{new-eqn}), we have
\begin{align}
	0=\sum_{n=0}^{\infty}f_nq^{2n+2}z^{n+2} - \sum_{n=0}^{\infty}f_nq^{n+2}z^{n+2} + \sum_{n=0}^{\infty}f_n z^{n+4}.
\end{align}
Thus we obtain
\begin{align}
	\frac{f_{n+2}}{f_n} = \frac{1}{q^{n+4}(1-q^{n+2})},
\end{align}
with initial conditions
\begin{align}
	& f_0 (q^2-q^2)=0, \\
	& f_1 (q^4-q^3)=0,
\end{align}
so we know $f_{2n+1}=0$ and 
\begin{align}
	f_{2n} = \frac{q^{-2}}{q^{2n}(1-q^{2n})} \cdots \frac{q^{-2}}{q^2(1-q^2)} \cdot f_0.
\end{align}
Taking $f_0=1$, we conclude that the solution of (\ref{new-eqn}) is as follows
\begin{align}
	F =  \sum_{n=0}^{\infty} \frac{\left(q^{\frac{n(n+1)}{2}}  \right)^{-2}}{(q^2;q^2)_n} (q^{-2}z)^n. 
\end{align}
\begin{Remark}
In general, for a root $x=c$, we consider the solution of the form
\begin{align}
	F = e_{q,c}(z) \sum_{n=0}^\infty f_n z^n,
\end{align}
as we did in the regular cases. If $c$ is a $q^{\mathbb{Z}}$-root, i.e. $c=q^n$, then 
\begin{align}
	e_{q,q^n}(z) = q^{\frac{n(n-1)}{2}}z^n.
\end{align}
\end{Remark}

For another root $x=0$, we can not construct a new solution from it. Indeed, if we consider a solution of the form
\begin{align}
	G= \sum_{n=0}^{\infty}g_nz^n,
\end{align}
substituting into (\ref{new-eqn}), we obtain
\begin{align}
	0 &= \sum_{n=0}^{\infty} g_n q^{2n-2}z^n - \sum^{\infty}_{n=0}g_nq^nz^n + \sum_{n=0}^{\infty}g_nz^n \\
	  &=\sum_{n=0}^{\infty} q^n(q^{n-2}-1)g_nz^n = \sum_{n=2}^{\infty}g_{n-2}z^n,
\end{align}
with initial conditions
\begin{align}
q^0(q^{-2}-1)g_0 &= 0,	\\
q^1(q^{-1}-q)g_1 &= 0,   \\
q^2(q^{0}- 1)g_2 &= 0,  
\end{align}
so $g_0=g_1=0$ and $g_2$ could be arbitrary. Then for $n \geq 2 $, we have
\begin{align}
\frac{g_{2n+2}}{g_{2n}} = \frac{1}{q^{2n+2}(1-q^{2n})},	
\end{align}
which is the same as root $x=q^2$ case. 

In summary, for the degree 2 difference equation
\begin{align} 
\left[z^2\sigma_q^2 - \sigma_q +1  \right]f=0,	
\end{align}
we can construct two solutions
\begin{align}
	 F_1(z) &= \sum^{\infty}_{n=0} \frac{\left( q^{\frac{n(n-1)}{2}} \right)^4}{(q^2;q^2)_n} (-1)^n z^{2n}, \\
	 F_2(z) &= \theta_q^{-2}(z) \sum_{n=0}^{\infty} \frac{\left(q^{\frac{n(n+1)}{2}}  \right)^{-2}}{(q^2;q^2)_n} (q^{-2}z)^n.
\end{align}

\end{Example}

\begin{Example}[Slope $\mu=-1/2$]
Consider following difference equation
\begin{align}
	\left[ z\sigma_q^2-1 \right]f=0, \label{equ-1/2}
\end{align}	
then the associated Newton polygon is as follows
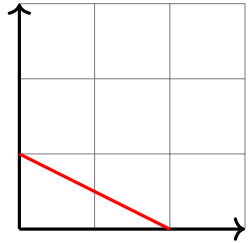
\begin{center}
\begin{tikzpicture}
\draw [xstep=1, ystep=1, draw=gray](0,0)grid(3,3);
\draw [->, very thick](0,0)--(3,0);
\draw [->, very thick](0,0)--(0,3);

\draw [very thick, draw=red](0,1)--(2,0); 
\end{tikzpicture}
\end{center}
with only one segment of slope $-\frac{1}{2}$, then consider the solution of the form
\begin{align}
	\theta_q^{-\frac{1}{2}}(z)F(z^{1/2}),
\end{align}
substituting into (\ref{equ-1/2}), we have
\begin{align}
	& z\left(q^{1\frac{1}{2}}z^{-1}\theta_q^{-\frac{1}{2}}\sigma_q^2F(z^{1/2}) \right)-\theta_q^{-\frac{1}{2}}F(z^{1/2})=0.
\end{align}
So we obtain a new difference equation
\begin{align}
	\left[ q^{-\frac{1}{2}}\sigma_q^2-1 \right]F(z^{1/2})=0,
\end{align}
with characteristic equation
\begin{align}
	q^{-\frac{1}{2}}x^2-1=0,
\end{align}
the two roots are $x=q^{\frac{1}{4}}$ and $x=-q^\frac{1}{4}$. Then one could consider the solutions of the form
\begin{align}
	& F_1 = e_{q,q^{\frac{1}{4}}}(z) \sum_{n=0}^{\infty}f_n z^{n/2}, \\
	& F_2 = e_{q,-q^{\frac{1}{4}}}(z) \sum_{n=0}^{\infty}f_n z^{n/2},
\end{align}
by a little bit computation, one could find these solutions are
\begin{align}
& F_1 = e_{q,q^{\frac{1}{4}}}(z), 	\\
& F_2 = e_{q,-q^{\frac{1}{4}}}(z).
\end{align}
Then the solutions of original equation are
\begin{align}
	\theta_q^{-\frac{1}{2}}(z)e_{q,q^{\frac{1}{4}}}(z), \quad \theta_q^{-\frac{1}{2}}(z)e_{q,-q^{\frac{1}{4}}}(z).
\end{align}
\end{Example}

\subsection{Difference equation for $\mathbb{P}^1$ with level structure ($l \geq 2$)} \label{sub-proj-irregular}
As we mentioned in subsection \ref{sub-proj}, the modified $I$-function of $\mathbb{P}^1$ with level structure is
\begin{align} 
\widetilde{I^{K, l}_{\mathbb{P}^1}} = 	P^{\ell_q(z)}\sum_{d=0}^{\infty} \frac{\left(P^{d}q^{\frac{d(d-1)}{2}} \right)^l z^d}{\prod_{k=1}^{d} (1-Pq^k)^2 },
\end{align}
satisfying the following difference equation
\begin{align}
\left[ (1-\sigma_q)^2 -z  \sigma_q^l  \right] \widetilde{I^{K,St}_{\mathbb{P}^1}} = 0.	\label{dif-l>2}
\end{align}
Here we consider $l >2$, then the above difference equation is irregular singular. 

The associated Newton polygon is as follows
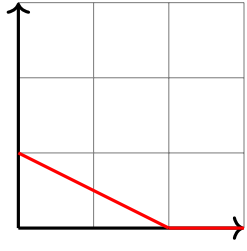
\begin{center}
\begin{tikzpicture}
\draw [xstep=1, ystep=1, draw=gray](0,0)grid(3,3);
\draw [->, very thick](0,0)--(3,0);
\draw [->, very thick](0,0)--(0,3);
\draw [very thick, draw=red](0,1)--(2,0); 
\draw [very thick, draw=red](2,0)--(3,0); 
\end{tikzpicture}
\end{center}
There are two segments, one is of slope $-1/(l-2)$ and another one is horizontal. The characteristic equation with respect to the horizontal segment is 
\begin{align}
	(1-x)^2=0,
\end{align}
the same as the $0 \leq l \leq 2$ (regular singular) case in subsection \ref{sub-proj}. Thus there are two solutions as before
\begin{align}
\sum_{d=0}^{\infty} \frac{\left(q^{\frac{d(d-1)}{2}} \right)^l z^d}{\prod_{k=1}^{d} (1-q^k)^2 },
\end{align}
and
\begin{align}
 \sum_{d=0}^\infty \frac{q^{\frac{ld(d-1)}{2}}}{\prod_{k=1}^d(1-q^k)^2} \left(\ell_q(z) - \sum_{k=1}^d \frac{2q^k}{1-q^k}  \right) + 	\sum_{d=0}^\infty \frac{ld \cdot q^{\frac{ld(d-1)}{2}}}{\prod_{k=1}^d(1-q^k)^2}.
\end{align}

For the segment of slope $\mu=-1/(l-2)$, we consider solutions of the form
\begin{align}
\theta_q^{-\frac{1}{(l-2)}}(z) F(z^{1/(l-2)},q).	
\end{align}
Let $Q=z^{1/(l-2)}$, $p=q^{1/(l-2)}$ and $\sigma_p = p^{Q \partial_Q}$. Then one find that $F(Q,p)$ satisfies 
\begin{align}
	\left[ \sigma_p^l - \sigma_p^2 +2Q - Q^2 \right] F(Q,p)=0.
\end{align}
With a new Newton polygon
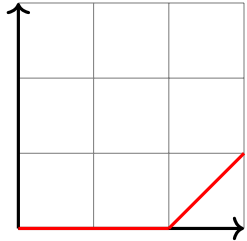
\begin{center}
\begin{tikzpicture}
\draw [xstep=1, ystep=1, draw=gray](0,0)grid(3,3);
\draw [->, very thick](0,0)--(3,0);
\draw [->, very thick](0,0)--(0,3);
\draw [very thick, draw=red](0,0)--(2,0); 
\draw [very thick, draw=red](2,0)--(3,1); 
\end{tikzpicture}
\end{center}
the characteristic equation is 
\begin{align}
x^2(x^{l-2}-1)=0.
\end{align}
So each $(l-2)$-th root of unity $\zeta$, we could construct a solution of the form
\begin{align}
e_{p,\zeta}(Q) \sum_{d=0}^\infty f_{d}(\zeta,p) Q^d, \ {\rm with} \quad f_0(\zeta,p)=1.	
\end{align}
Then we obtain a relation of the coefficient $f_{d}(\zeta,p)$ as follows,
\begin{align}
	\sum_{d \geq 0} \left[ f_d(\zeta,p)\cdot \zeta^2(p^{ld}-p^{2d})Q^d + 2f_{d}(\zeta,p) Q^{d+1} - f_{d}(\zeta,p)Q^{d+2} \right] = 0.
\end{align}
Thus
\begin{align}
	\zeta^2(p^{ld}-p^{2d})f_d(\zeta,p) = 2f_{d-1}(\zeta,p) - f_{d-2}(\zeta,p), \quad d \geq 1, 
\end{align}
where we set $f_{-1}(\zeta,p) = 0$. Thus we construct $l$ solutions for the difference equation (\ref{dif-l>2}).

\subsection{Difference equation for quintic 3-fold}
As we introduced in the introduction, the modified $I$-function of quintic is 
\begin{align}
\widetilde{I^K_X}(z,q)= P^{\ell_q(z)}\sum_{d=0}^{\infty} \frac{\prod_{k=1}^{5d}(1-P^5q^k)}{\prod_{k=1}^{d}(1-Pq^k)^5} z^d,
\end{align}
satisfying the following difference equation of degree 25:
\begin{align} 
\left[	(1-\sigma_q)^5 - z \prod_{k=1}^5 (1-q^{k}\sigma_q^5) \right] \widetilde{I^K_X}(q,z)=0,     
\end{align}
whose associated Newton polygon is
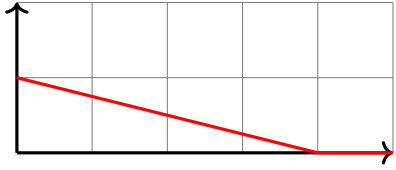
\begin{center}
\begin{tikzpicture}
\draw [xstep=1, ystep=1, draw=gray](0,0)grid(5,2);
\draw [->, very thick](0,0)--(5,0);
\draw [->, very thick](0,0)--(0,2);
\draw [very thick, draw=red](0,1)--(4,0); 
\draw [very thick, draw=red](4,0)--(5,0); 
\end{tikzpicture}
\end{center}
There are two segments, one is of slope $\mu=-1/20$ and another one is horizontal. We could do the same as the last subsection, for the non-horizontal segment, we have

\begin{Proposition}[\cite{Wen}, Proposition 3.1.] Let $\xi$ be the 20th root of unity and $p$ be $q^{1/20}$. Setting $Q=z^{1/20}$ and $\sigma_p=p^{Q\partial_Q}$. There are 20 solutions associated to the segment of slope $\mu=-1/20$ of the form
\begin{align}
	e_{p,\xi p^{-1/2}}F(Q) = e_{p,\xi p^{-1/2}}\sum_{n \geq 0}f_nQ^n 
\end{align}
where $F(Q)$ satisfies
\begin{align}
	&\left[ (z-\xi p^{-\frac{9}{2}}\sigma_p)(z-\xi p^{-\frac{7}{2}}\sigma_p)(z-\xi p^{-\frac{5}{2}}\sigma_p)(z-\xi p^{-\frac{3}{2}}\sigma_p)(z-\xi p^{-\frac{1}{2}} \sigma_p)    \right.  \\
	&- \left.  (z^5-\xi^5p^{-\frac{25}{2}}\sigma_p^5)(z^5-\xi^5 p^{-\frac{15}{2}}\sigma_p^5)(z^5-\xi^5p^{-\frac{5}{2}}\sigma_p^5)(z^5-\xi^5p^{\frac{5}{2}}\sigma_p^5)(z^5-\xi^5p^{\frac{15}{2}}\sigma_p^5)   \right] F(Q)=0
\end{align}	
\end{Proposition}

\begin{Remark}
In this notes, we require $|q|>1$. However in \cite{Wen}, in order to do the analytic continuation, it needs to require $|q|<1$ for the convergence reason. And see \cite{GS21} for additional discussion on q-deformed Picard-Fuchs equation and Frobenius method for quintic 3-fold.
\end{Remark}

\subsection{Convergent solutions for irregular cases}
In this subsection, we prove convergence of certain solutions in irregular singular cases. Here we follow \cite{Adams, Adams2, HSS16}.

Let $P=a_0 + \ldots + a_n \sigma_q^n$ be the standard form, suppose $a_0a_n \neq 0$ and at least one $a_i(0)$ is nonzero. Setting
\begin{align}
a_k(z)= \sum_{j=0}^\infty a_{kj} z^j,	
\end{align}
and
\begin{align}
P_j (\sigma_q) = \sum_{k=0}^n a_{kj} \cdot \sigma_q^k.	
\end{align}
Then we have

\begin{Lemma}[\cite{Adams} Chapter 3, p. 202.]
Assume that the lowest slope of $P$ 	is 0 and $P_0(1)=0$, $P_0(q^k) \neq 0$ for $ \forall \ k \geq 1 $. Then the unique formal solution of 
\begin{align}
	P \cdot f(z,q) = \left[ a_0 + \ldots + a_n \sigma_q^n \right]f(z,q)=0
\end{align} 
in the form 
\begin{align}
f =1 + f_1 z + \ldots \in \mathbb{C}[\![ z]\!]	
\end{align}
converges.
\end{Lemma}
\begin{proof}
	It almost follows the procedure in section \ref{reg-convergence}, here we gives a brief proof. There exist a constant $A > 0$ such that
\begin{align}
 |P_0(q^k)| \geq A |q |^{kn},\quad \forall k \geq 1,.
\end{align}
Let $R$ be strictly bounded above by the radius of convergence of $a_0(z), \ldots, a_n(z)$. Then there is a $B > 0$ such that
\begin{align}
\forall i \in \{0, \ldots, n \}, \ \forall j \geq 0, \ |a_{ij}| \leq BR^{-j}	.
\end{align}
Then 
\begin{align}
	|P_j(q^i)| \leq DR^{-j} |q|^{ni}.
\end{align}
From the recursive relation, we have for $m \geq 1$
\begin{align}
|f_m| \leq \frac{D}{A} \sum_{i=0}^{m-1} \frac{R^{-(m-i)}|q|^{ni}}{|q|^{nm}} |f_i|,	
\end{align}
quotient by $(R^{-1}|q|^{-n})^m$ on both side, we obtain
\begin{align}
\frac{|f_m|}{(R^{-1}|q|^{-n})^m} \leq \frac{D}{A} \sum_{i=0}^{m-1} \frac{|f_i|}{(R^{-1}|q|^{-n})^i},	
\end{align}
i.e.,
\begin{align}
|f_m| \leq  \frac{mD}{A (R|q|^n)^m} |f_0|.	
\end{align}
The convergence of $f$ follows.
\end{proof}

\begin{Remark} 
\begin{itemize}
\item The above argument fails if $0$ is not	 the lowest slope of $P$. For example, one could check our Example \ref{slope=-1} and \ref{slope=-2} that the solutions associate with horizontal segment do not converge.
\item If the lowest slope of $P$ is $\mu \in \mathbb{Z}_{<0} $, then there is a solution of the form
\begin{align}
 \theta_q^{\mu}(z) F(z,q)= \theta_q^{\mu}(z)\sum_{d=0}^\infty f_d(q) \cdot z^d	,
\end{align}
and $F(z,q)$ is analytic at the origin. The proof is almost the same, since the prefactor $\theta_q^{\mu}(z)$ will make the lowest slope of the new difference equation for $F(z,q)$ to be 0.
\end{itemize}
\end{Remark}


\end{document}